\documentclass[11pt]{article} 
\usepackage{amssymb,amsmath,amsthm}
\usepackage[letterpaper,hmargin=1.0in,vmargin=1.0in]{geometry}
\usepackage{graphics} \usepackage{tikz}
\usepackage{color} \usepackage[boxed]{algorithm}
\newtheorem{theorem}{Theorem}[section]
\newtheorem{lemma}[theorem]{Lemma}
\newtheorem{proposition}[theorem]{Proposition}

 \newcommand{\E}{\mathbb E}

\newcommand{\HY}{\mathcal H}

\begin{document} 
	\title{Many cliques in $H$-free subgraphs of random graphs\\
	} \author{ Noga Alon\thanks{Sackler School of Mathematics and Blavatnik
		School of Computer Science, Tel Aviv University, Tel Aviv 69978, Israel
		and \color{black} CMSA, Harvard University, Cambridge, MA 02138,
		USA\color{black}. Email: nogaa@tau.ac.il. Research supported in part by
		an ISF grant and by a GIF grant. }\and Alexandr
	Kostochka\thanks{Department of Mathematics, University of Illinois at
		Urbana-Champaign, Urbana, IL 61801, USA and Sobolev Institute of Mathematics,  Novosibirsk
		630090, Russia. Email: kostochk@math.uiuc.edu. Research of this author
		is supported in part by NSF grant   DMS-1600592 and by grants
		15-01-05867 and 16-01-00499  of the Russian Foundation for Basic
		Research.} \and Clara Shikhelman \thanks{Sackler School of Mathematics,
		Tel Aviv University, Tel Aviv 69978, Israel. Email:
		clarashk@mail.tau.ac.il. Research supported in part by an ISF grant. }
}

\setlength{\parskip}{1ex plus 0.5ex minus 0.2ex}

\maketitle

\begin{abstract} For two fixed graphs $T$ and $H$ let $ex(G(n,p),T,H)$
	be the random variable counting the maximum number of copies of $T$ in
	an $H$-free subgraph of the random graph $G(n,p)$. We show that for the
	case $T=K_m$ and $\chi(H)> m$ the behavior of $ex(G(n,p),K_m,H)$ depends
	strongly on the relation between $p$ and \color{black}
	$m_2(H)=\max_{H'\subseteq H, |V(H')|'\geq 3}\left\{
	\frac{e(H')-1}{v(H')-2} \right\}$. \color{black}
	
	When $m_2(H)> m_2(K_m)$ we prove that with high probability, depending
	on the value of $p$, either one can maintain almost all copies of
	$K_m$, or it is asymptotically best to take a $\chi(H)-1$ partite
	subgraph of $G(n,p)$.  The transition between these two behaviors
	occurs at $p=n^{-1/m_2(H)}$. When $m_2(H)< m_2(K_m)$ we show that the
	above cases still exist, however for $\delta>0$ small at
	$p=n^{-1/m_2(H)+\delta}$ one can typically still keep most of the
	copies of $K_m$ in an $H$-free subgraph of $G(n,p)$. Thus, the
	transition between the two behaviors in this case occurs at some $p$
	significantly bigger than $n^{-1/m_2(H)}$.
	
	To show that the second case is not redundant we present a construction
	which  may be of independent interest. For each $k \geq 4$ we construct
	a family of $k$-chromatic graphs $G(k,\epsilon_i)$ where
	$m_2(G(k,\epsilon_i))$ tends to $\frac{(k+1)(k-2)}{2(k-1)}<
	m_2(K_{k-1})$ as $i$ tends to infinity. This is tight for all values of
	$k$ as for any $k$-chromatic graph $G$,
	$m_2(G)>\frac{(k+1)(k-2)}{2(k-1)}$. \vspace{0.5cm}
	
\end{abstract}

\section{Introduction}

The well known  \color{black}Tur\'an \color{black} function, denoted
$ex(n,H)$, counts the maximum number of edges in an $H$-free subgraph of
the complete graph on $n$ vertices (see for example \cite{Si} for a
survey). A natural generalization of this question is to  change the
base graph and instead of taking a subgraph of the complete graph
consider a subgraph of a random graph. More precisely let $G(n,p)$ be
the random graph on $n$ vertices where each edge is chosen randomly and
independently with probability $p$. Let $ex(G(n,p),H)$ denote the random
variable counting the maximum number of edges in an $H$-free subgraph of
$G(n,p)$.

The behavior of $ex(G(n,p),H)$  is studied in 
\cite{FR}, and  additional results appear in \cite{RR}, \cite{KLR},
\cite{HKL95}, \cite{HKL96} and more. 
%
Taking an extremal graph $G$ which is $H$-free on $n$ vertices with
$ex(n,H)$ edges and then keeping each edge of $G$ randomly and
independently with probability $p$ shows that w.h.p., that is, with
probability tending to 1 as $n$ tends to infinity,

$$ex(G(n,p),H)\geq (1+o(1)) ex(n,H)p.$$

In \cite{KLR} Kohayakawa, \L uczak and R\"odl  and in \cite{HKL95}
Haxell, Kohayakawa and \L uczak conjectured that the opposite inequality
is asymptotically valid  for values of $p$ for which each edge in
$G(n,p)$ takes part in a copy of $H$.

This conjecture was proved by Conlon and Gowers in \cite{CG}, for the
balanced case, and  by Schacht in \cite{Sc} for general graphs (see also
\cite{BMS} and \cite{ST}). 
Motivated by the condition that each edge is in a copy of $H$, define
the {\em 2-density} of a graph $H$, denoted by $m_2(H)$, to be

$$m_2(H)=\max_{H'\subseteq H, v(H')\geq 3}\left\{
\frac{e(H')-1}{v(H')-2} \right\}.$$ The Erd\H{o}s-Simonovits-Stone
theorem states that
$ex(n,H)=\binom{n}{2}\left(1-\frac{1}{\chi(H)-1}+o(1)\right)$, and so
the theorem proved in the papers above, restated in simpler terms is the
following \color{black}
\begin{theorem}[\cite{CG},\cite{Sc}]\label{thm:RegRandTur} For any fixed
	graph $H$ the following holds w.h.p.\ $$ex(G(n,p),H)=\begin{cases}
	\big(1-\frac{1}{\chi(H)-1}+o(1)\big)\binom{n}{2} p & \text{for } p \gg
	n^{-1/m_2(H)} \\ (1+o(1))\binom{n}{2}p & \text{for } p \ll n^{-1/m_2(H)}
	\end{cases} $$\end{theorem} \noindent \color{black} where here and in
what follows we write $f(n)\gg g(n)$ when $\underset{n\to\infty}{\lim}
\frac{f(n)}{g(n)}=\infty$. \color{black}

Another generalization of the classical \color{black}Tur\'an
\color{black}question is to ask for the maximum number of copies of a
graph $T$ in an $H$-free subgraph of the complete graph on $n$ vertices.
This function, denoted $ex(n,T,H)$, is studied in \cite{ASh} and in some
special cases in the references therein. Combining both generalizations
we define the following. For two graphs $T$ and $H,$  let
$ex(G(n,p),T,H)$ be the random variable whose value is the maximum
number of copies of $T$ in an $H$-free subgraph of $G(n,p)$. Note that
as before the expected value of $ex(G(n,p),T,H)$ is at least $
ex(n,T,H)p^{e(T)}$ for any $T$ and $H$.

In \cite{ASh} it is shown that for any $H$ with $\chi(H)=k>m$,
$ex(n,K_m,H)=(1+o(1))\binom{k-1}{m}(\frac{n}{k-1})^m$. This motivates
the following question analogous to the one answered in Theorem
\ref{thm:RegRandTur}: \color{black}{\em For which values of $p$  is it
	true that $ex(G(n,p),K_m,H)=(1+o(1))\binom{k-1}{m}(\frac{n}{k-1})^m
	p^{\binom{m}{2}}$ w.h.p.? } \color{black}

%
We show that the behavior of $ex(G(n,p),K_m,H)$ depends strongly on the
relation between $m_2(K_m)$ and $m_2(H)$. When $m_2(H)>m_2(K_m)$ there
are two regions in which the random variable behaves differently. If $p$
is much smaller than $n^{-1/m_2(H)}$ then the $H$-free subgraph of
$G\sim G(n,p)$ with the maximum number of copies of $K_m$ has w.h.p.
most of the copies of $K_m$ in $G$ as only a negligible number of edges
take part in a copy of $H$. When $p$ is much bigger than $n^{-1/m_2(H)}$
we can no longer keep most of the copies of $K_m$ in an $H$-free
subgraph and it is asymptotically best to take a  {$(k-1)$-}partite
subgraph of $G(n,p)$. \color{black} The last part also holds when
$m_2(H)=m_2(K_m)$. \color{black} Our first theorem is the following:

\color{black} \begin{theorem} \label{thm:m_2InEasyOrder} Let $H$ be a
	fixed graph with $\chi(H)=k>m$. If $p$ is such that
	$\binom{n}{m}p^{\binom{m}{2}}$ tends to infinity as $n$ tends to
	infinity then w.h.p.\
	
	$$ex(G(n,p),K_m,H)=\begin{cases} (1+o(1))\binom{k-1}{m}
	(\frac{n}{k-1})^m p^{\binom{m}{2}} & \text{for } p\gg n^{ -1/m_2(H)}
	\text{ \color{black} provided \color{black} }m_2(H)\geq m_2(K_m) \\
	(1+o(1))\binom{n}{m}p^{\binom{m}{2}} & \text{for }p\ll n^{ -1/m_2(H)}
	\text{ \color{black} provided \color{black} }m_2(H)> m_2(K_m)
	\end{cases} $$ \end{theorem} \color{black}

Theorem \ref{thm:m_2InEasyOrder} is valid when $m_2(H)>m_2(K_m)$. What
about graphs $H$ with $\chi(H)=k>m$ as before but $m_2(H)<m_2(K_m)$?
\color{black} Do such graphs $H$ exist at all? \color{black}

A graph $H$ is $k$-critical if $\chi(H)=k$ and  for any subgraph
$H'\subset H$, $\chi(H')<k$. In \cite{KY} Kostochka and Yancey show that
if $k \geq 4$ and $H$ is $k$-critical, then

$$ e(H) \geq  \left\lceil
\frac{(k+1)(k-2)v(H)-k(k-3)}{2(k-1)}\right\rceil.$$ This implies that
for every $k$-critical $n$-vertex graph $H$,

\begin{equation}\label{ky} { \frac{e(H)-1}{v(H)-2}\geq
	\frac{(k+1)(k-2)n-k(k-3)-2(k-1)}{2(k-1)(n-2)}>\frac{(k+1)(k-2)}{2(k-1)}.} \end{equation} Therefore for any $H$ with $\chi(H)=k$ one has $$m_2(H)>\frac{(k+1)(k-2)}{2(k-1)}.$$  This \color{black} implies \color{black} that Theorem \ref{thm:m_2InEasyOrder} 
covers any graph $H$ for which $\chi(H)\geq m+2$, since
$m_2(K_m)=\frac{m+1}{2}$.

When $\chi(H)=m+1$ the situation is more complicated. Before
investigating the function $ex(G(n,p),K_m,H)$ for these graphs we show
that the case $m_2(H)<m_2(K_m)$ and $\chi(H)=m+1$ is not redundant. To
do so we prove the following theorem, which may be of independent
interest. The theorem strengthens the result in \cite{ABGKM} for $m=3$,
expands it to any $m$, and by \cite{KY} 
it is tight. \begin{theorem} \label{thm:existOfH} For every fixed $k\geq
	4$ and $\epsilon>0$ there exist infinitely many $k$-chromatic graphs
	$G(k,\epsilon)$ with $$m_2(G(k,\epsilon))\leq
	(1+\epsilon)\frac{(k+1)(k-2)}{2(k-1)}.$$ \end{theorem}

This theorem shows that there are infinitely many $m+1$ chromatic graphs
$H$ with $m_2(H)<m_2(K_m)$. \color{black} For these graphs there are
three regions of interest for the value of $p$: $p$ much bigger than
$n^{-1/m_2(K_m)}$, $p$ much smaller than $n^{-1/m_2(H)}${,} and $p$ in
the middle range.\color{black}

One might suspect that as before the function $ex(G(n,p),K_m,H)$ will
change its behavior at $p=n^{-1/m_2(H)}$ but this is no longer the case.
We prove that for some graphs $H$ when $p$ is slightly bigger than
$n^{-1/m_2(H)}$ we can still take w.h.p.\ an $H$-free subgraph of
$G(n,p)$ that contains most of the copies of $K_m$:

\begin{theorem} \label{thm:pAlmostM_2(H)} Let $H$ be a graph such that
	$\chi(H)=m+1\geq 4$, $m_2(H)<c$ for some $c<m_2(K_m)$ and there exists
	$H_0\subseteq H$ for which $\frac{e(H_0)-1}{v(H_0)-2}=m_2(H)$ and
	$v(H_0)>M(m,c)$ \color{black} where $M(m,c)$ is \color{black} large
	enough. If $p\leq n^{-\frac{1}{m_2(H)}+\delta}$ for
	$\delta:=\delta(m,c)>0$ small enough  and $\binom{n}{m}p^{\binom{m}{2}}$
	tends to infinity as $n$ tends to infinity{,} then w.h.p.\
	$$ex(G(n,p),K_m,H)=(1+o(1))\binom{n}{m} p^{\binom{m}{2}}.$$
\end{theorem}

\color{black}On the other hand, we prove that for big  enough
\color{black} values of $p$ one cannot find an $H$-free subgraph of
$G(n,p)$ with $(1+o(1))\binom{n}{m} p^{\binom{m}{2}}$ copies of $K_m$
and it is asymptotically best to take a $k-1$-partite subgraph of
$G(n,p)$.

As an example we show that the theorem above can be applied to the
graphs constructed in Theorem \ref{thm:existOfH}.

\begin{lemma} \label{lem:bigV} For every two integers $k$ and $N$ there
	is $\epsilon>0$ small enough such that $v(G_0(k,\epsilon))>N$, where
	$G_0(k,\epsilon)$ is a subgraph of $G(k,\epsilon)$ for which
	$\frac{e(G_0(k,\epsilon))-1}{v(G_0(k,\epsilon))-2}=m_2(G(k,\epsilon))$.
\end{lemma} %
%
%
%
%
The rest of the paper is organized as follows. In Section 2 we establish
some general results for $G(n,p)$. \color{black} In Section 3 we prove
Theorem \ref{thm:m_2InEasyOrder}. 
In Section 4 we describe the construction of sparse graphs with a given
chromatic number and prove Theorem \ref{thm:existOfH}. In Section 5 we
prove Theorem \ref{thm:pAlmostM_2(H)} and Lemma \ref{lem:bigV}.
\color{black} We finish with some concluding remarks and open problems
in Section 6.

\section{Auxiliary Results}

We need the following well known Chernoff bounds on the upper and lower
tails of the binomial distribution (see e.g. \cite{AS}, \cite{MR})
\begin{lemma} \label{lem:Chernoff} Let $X\sim Bin(n,p)$ then
	\begin{enumerate} \item
		$\mathbb{P}(X<(1-a)\mathbb{E}X)<e^{\frac{-a^2\mathbb{E}X}{2}}$ for
		$0<a<1$  \label{lem:ChernoffLower} \item
		$\mathbb{P}(X>(1+a)\mathbb{E}X)<e^{\frac{-a^2\mathbb{E}X}{3}}$ for
		$0<a<1$  \label{lem:ChernoffUppe} \item
		$\mathbb{P}(X>(1+a)\mathbb{E}X)<e^{\frac{-a\mathbb{E}X}{3}}$ for $a>1$
		\label{lem:ChernoffUppeBigA} \end{enumerate} \end{lemma} %

The following known result is used a few times \begin{theorem}[see,
	e.g., Theorem 4.4.5 in \cite{AS}] \label{lem:ExpectedNumLemma} Let $H$
	be a fixed graph. For every subgraph $H'$ of $H$ (including $H$ itself)
	let $X_{H'}$ denote the number of copies of $H'$ in $G(n,p)$. Assume $p$
	is such that $\mathbb{E} [X_{H'}]\to \infty$ for every $H'{\subseteq
		H}$. Then w.h.p.\ $$X_H=(1+o(1))\mathbb{E} [X_H].$$ \end{theorem}

In addition we prove technical lemmas to be used in Sections
\ref{sec:EasyOrder} and \ref{sec:DifOrder}. \color{black} From here on
for two graphs $G$ and $H$ we denote by $\mathcal{N}(G,H)$ the number of
copies of $H$ in  $G$.\color{black}

\begin{lemma} \label{lem:CopiesOnEdge} Let $G\sim G(n,p)$ with $p\gg
	n^{-1/m_2(K_m)}$ then w.h.p.\ \begin{enumerate} \item Every set of
		$o(pn^2)$ edges takes part in  $o\big(\mathcal{N}(G,K_m)\big)$ copies of
		$K_m$, \item For every $\epsilon>0$ small enough  every set of
		$n^{-\epsilon} pn^2$ edges takes part in at most
		$n^{-\epsilon/3}\mathcal{N}(G,K_m)$ copies of $K_m$. \end{enumerate}
\end{lemma}

\begin{proof} Let $G\sim G(n,p)$ and  let $X$  be the random variable
	counting the number of copies of $K_m$ on a randomly chosen edge of
	$G(n,p)$. First we show that $\E[X^2]\leq O(\E^2[X])$. Given an edge let
	$\{A_1,...A_{l}\}$ be all the possible copies of $K_m$ using this edge
	in $K_n$ and let $|A_i\cap A_j|$ be the number of vertices the copies
	share. Let $X_{A_i}$ be the indicator of the event $A_i\subset G$. Then
	$X=\sum X_{A_i}$ and we get that \begin{align*}
	\E^2[X]&=\color{black}(\sum \E
	[X_{A_i}])^2\color{black}=\Theta([n^{m-2}p^{\binom{m}{2}-1}]^2)\\
	\E[X^2]&=\E[\sum_{k=2}^{m} \sum_{|A_i\cap A_j|=k}X_{A_i}X_{A_j}]\\ &\leq
	\sum^{m}_{k=2} n^{2m-k-2}p^{\binom{m}{2}+\binom{m-k}{2}+(m-k)k-1}
	\end{align*}
	
	\color{black}Put \color{black}
	$S_k=n^{2m-k-2}p^{\binom{m}{2}+\binom{m-k}{2}+(m-k)k-1}$ and note that
	\color{black}$S_2=\Theta(\E^2[X])$\color{black}. Furthermore, for any
	$2< k \leq m$ the following holds
	$S_2/S_k=n^{k-2}p^{\binom{k}{2}-1}\overset{n\to
		\infty}{\longrightarrow}\infty$ as $p\gg n^{-1/m_2(K_m)}\geq
	n^{-1/m_2(K_k)}$ and from this \begin{equation}\label{eq:expSq}
	\E[X^2]\leq O(\E^2[X]). \end{equation} \color{black}(Note that in fact
	$\E[X^2]=(1+o(1))\E^2[X]$ but the above estimate suffices for our
	purpose here)\color{black}
	
	Let $M=\mathcal{N}(G,K_m)$. 
	To prove the first part assume towards a contradiction that there is a
	set of edges, $E_0\subseteq E(G)$, which is of size $o( n^2 p)$ and
	that there exists $c>0$ such that there are $cM$ copies of $K_m$
	\color{black} containing \color{black} at least one edge from it.
	
	On one hand,  $\E^2[X]=[M\binom{m}{2}\frac{1}{e(G)}]^2$. On the other
	hand \color{black} by Jensen's inequality \color{black}
	\begin{equation*} \E[X^2]\geq \E[X^2 \mid e\in E_0]\mathbb{P}[e\in
	E_0]\color{black}\geq \color{black}\Big(\frac{cM}{|E_0|}\Big)^2
	\cdot\frac{|E_0|}{e(G)}=\Big(\frac{M\binom{m}{2}}{e(G) } \Big)^2
	\frac{c^2}{\binom{m}{2}^2}\frac{e(G)}{|E_0|}=\omega(\E^2[X])
	\end{equation*} \color{black}where \color{black} the last equality
	\color{black} holds \color{black} as $|E_0|=o(e(G))$. This is a
	contradiction to (\ref{eq:expSq}) and so the first part of the Lemma
	holds.
	
	For the second part assume there is a set $E_0$ such that
	$|E_0|=n^{-\epsilon}pn^2$ and the set of copies of $K_m$ using
	\color{black} edges \color{black} of $E_0$ is of size at least
	$n^{-\epsilon/3}M$. Note that w.h.p.\ $e(G)\geq\frac{1}{4}n^2 p$.
	Repeating the calculation above we get that \begin{equation*}
	\E[X^2]\geq\E[X^2 \mid e\in E_0]\mathbb{P}[e\in
	E_0]=(\frac{n^{-\epsilon/3}M}{n^{-\epsilon} e(G)})^2\cdot
	\frac{n^{-\epsilon}}{4}=\frac{M^2}{e(G)^2}\frac
	{n^{\epsilon/3}}{4}=\omega(\E^2[X]) \end{equation*} which is again a
	contradiction, and thus the second part of the lemma
	\color{black}holds.\color{black} \end{proof}

\color{black} \color{black} \begin{lemma} \label{lem:oneK_mPerEdge} Let
	$G\sim G(n,p)$ for $p=n^{-a}$ with $-a<-1/m_2(K_m)$. Then w.h.p.\ the
	number of copies of $K_m$ sharing an edge with other copies of $K_m$ is
	$o(n^mp^{\binom{m}{2}})$. \end{lemma} \color{black}

\begin{proof} First note that
	$n^{m-2}p^{\binom{m}{2}-1}=(np^{(m+1)/2})^{m-2}=n^{-\alpha(m-2)}$ for
	some $\alpha>0$. The expected number of \color{black} pairs of
	\color{black} copies of $K_m$ sharing $a$ vertices, where $m-1\geq
	a\geq2$ is at most
	
	\begin{align*}
	n^{2m-a}p^{\binom{m}{2}+\binom{m-a}{2}+(m-a)a}&=n^mp^{\binom{m}{2}}\cdot (np^{\frac{m+a-1}{2}})^{(m-a)}\\ &<n^mp^{\binom{m}{2}} np^{\frac{m+1}{2}}\\ &=n^mp^{\binom{m}{2}}n^{-\alpha}. \end{align*} Here we used the fact that $np^{\frac{m+1}{2}}<1$ and $p<1$.
	
	Using Markov's inequality we get that the probability that $G$ has more
	than $\color{black}2\color{black}n^mp^{\binom{m}{2}}n^{-\alpha/2}$
	copies of $K_m$ sharing an edge is no more than $n^{-\alpha/2}$.
\end{proof}

\color{black} \section{Proof of Theorem 1.2} \label{sec:EasyOrder}
\color{black}

To prove Theorem \ref{thm:m_2InEasyOrder}{,}  we prove three lemmas for
three ranges of values of $p$ using different approaches. \color{black}
Lemmas \ref{lem:BigValuesOfp} and \ref{lem:SmallValuesOfp} are stated in
a more general form as they are also used in Section \ref{sec:DifOrder}.
An explanation on how the lemmas prove Theorem \ref{thm:m_2InEasyOrder}
follows after the statements. \color{black}


\begin{lemma}\label{lem:BigValuesOfp} Let $H$ be a fixed graph with
	$\chi(H)=k>m$ and let \color{black} $p\gg  \max\{n^{-\frac{1}{m_2(H)}},
	n^{-\frac{1}{m_2(K_m)}}\}$ \color{black}.  Then
	$$ex(G(n,p),K_m,H)=(1+o(1))\binom{k-1}{m}\Big(\frac{n}{k-1}\Big)^m
	p^{\binom{m}{2}}.$$ \end{lemma} \color{black}

\begin{lemma}\label{lem:SmallValuesOfp} Let $H$ be a fixed graph with
	$\chi(H)=k>m$, \color{black} let $p<\min \{n^{-\frac{1}{m_2(H)}-\delta},
	n^{-\frac{1}{m_2(K_m)}-\delta} \}$ for some fixed $\delta>0$
	\color{black} and assume $n^m p^{\binom{m}{2}}$ tends to infinity as $n$
	tends to infinity{. T}hen
	$$ex(G(n,p),K_m,H)=(1+o(1))\binom{n}{m}p^{\binom{m}{2}}.$$ \end{lemma} %

\begin{lemma}\label{lem:MidRangeForEasyOrd} Let $H$ be a fixed graph
	with $\chi(H)=k>m$ and let $n^{-1/m_2(K_m)-\epsilon}<\color{black}p\ll
	n^{-1/m_2(H)}\color{black}$  where $\epsilon>0$ is sufficiently small.
	Then $$ex(G(n,p),K_m,H)=(1+o(1))\binom{n}{m}p^{\binom{m}{2}}.$$
\end{lemma} \color{black}

\color{black} Lemma \ref{lem:BigValuesOfp} takes care of the first part
of Theorem \ref{thm:m_2InEasyOrder}. If $m_2(H)\geq m_2(K_m)$ then
$n^{-1/m_2(H)}\geq n^{-1/m_2(K_m)}$ and this lemma covers values of $p$
for which $p\gg n^{-1/m_2(H)}$.

For the second part of Theorem \ref{thm:m_2InEasyOrder} we have Lemmas
\ref{lem:SmallValuesOfp}  and \ref{lem:MidRangeForEasyOrd}. If
$m_2(H)>m_2(K_m)$ Lemma \ref{lem:SmallValuesOfp} covers values of $p$
for which $p<n^{-1/m_2(K_m)-\delta}$ and Lemma
\ref{lem:MidRangeForEasyOrd} covers the range
$n^{-1/m_2(K_m)-\epsilon}<p\ll n^{-1/m_2(H)}$. Choosing
$\epsilon>\delta$ makes sure we do not miss values of  $p$.
\color{black}

We mostly focus on the proof of Lemma \ref{lem:BigValuesOfp}, as the
other two are simpler. Lemmas \ref{lem:BigValuesOfp} and
\ref{lem:SmallValuesOfp} are also relevant for the case
$m_2(H)<m_2(K_m)$, and are used again in Section \ref{sec:DifOrder}. For
the proof of Lemma \ref{lem:BigValuesOfp} we need several tools.

\begin{lemma} \label{lem:CopiesOfK_mInSided} Let $G$ be a $k$-partite
	complete graph with each side of size $n$, let $p\in[0,1]$ and let $G'$
	be a random subgraph of $G$ where each edge is chosen randomly and
	independently with probability $p$. If $n^mp^{\binom{m}{2}}$ goes to
	infinity together with $n$ then the number of copies of $K_m$ for $m<k$
	with each vertex in a different $V_i$ is w.h.p.\
	$$(1+o(1))\binom{k}{m}n^mp^{\binom{m}{2}} . $$ \end{lemma}

To  prove the lemma, we use the following concentration result:

\begin{lemma}[see, e.g., Corollary 4.3.5 in \cite{AS}]
	\label{lem:IndEventsConst} Let $X_1, X_2,...,X_r$ be indicator random
	variables for events $A_i$, and let $X=\sum_{i=1}^{r} X_i$. Furthermore
	assume $X_1,...,X_r$ are symmetric (i.e. for every $i\ne j$ there is a
	measure preserving mapping of the probability space that sends event
	$A_i$ to $A_j$). Write $i\sim j$ for $i\ne j$ if the events $A_i$ and
	$A_j$ are not independent. Set $\Delta^*=\sum_{i\sim j}\mathbb{P}(A_j |
	A_i)$ for some fixed $i$. If $\E[X]\to \infty$ and $\Delta^*=o(\E[X])$
	then $X=(1+o(1))\E(X)$. \end{lemma}

\begin{proof}[Proof of lemma  \ref{lem:CopiesOfK_mInSided}] The expected
	number of copies of $K_m$ in $G'$ is $(1+o(1))\binom{k}{m}n^m
	p^{\binom{m}{2}}${.  S}o we only need to show that it is indeed
	concentrated around its expectation. To do so we use Lemma
	\ref{lem:IndEventsConst}.

	Let $A_i$ be the event that a specific copy of $K_m$ appears in $G'$,
	and $X_i$ {be}  its indicator function. Clearly the number of copies of
	$K_m$ in $G'$ is $X=\sum X_i$. In this case $i\sim j$ if the
	{corresponding copies of $K_m$} share edges. We write $i\cap j=a$ if
	the two copies share exactly $a$ vertices. It is clear that the
	variables $X_i$ are symmetric. By the definition in the lemma{,}
	
	\begin{align*} \Delta^*&=\sum_{i\sim j}\mathbb{P}(A_j | A_i)\\
	&=\sum_{2\leq  a\leq m-1}\sum_{i\cap j=a}\mathbb{P}(A_j | A_i)\\ &\leq
	\sum_{2\leq  a\leq m-1} \binom{m}{a} \binom{k-a}{m-a}
	n^{m-a}p^{\binom{m-a}{2}+(m-a)a}\\ &=o(\binom{k}{m}n^m
	p^{\binom{m}{2}}){.} \end{align*}
	
	The last inequality \color{black} holds \color{black} as
	$n^{m}p^{\binom{m}{2}}=n^{m-a}p^{\binom{m-a}{2}+(m-a)a}\cdot
	n^{a}p^{\binom{a}{2}}$ and
	$n^{a}{p^{\binom{a}{2}}}=(np^{\frac{a-1}{2}})^a$ tends to infinity as
	$n$ tends to infinity  for $a<m$. \end{proof}

\color{black}To prove the upper bound in Lemma \ref{lem:BigValuesOfp} we
use \color{black} a standard technique for estimating the number of
copies of a certain graph inside another. \color{black}This is done by
applying \color{black} Szemeredi's regularity lemma and then a relevant
counting lemma. The regularity lemma allows us to find an equipartition
of any graph into a constant number of sets $\{V_i\}$, such that most of
the pairs of sets $\{V_i,V_j\}$ are regular (i.e. the densities between
large subsets of sets $V_i$ and $V_j$ do not deviate by more than
$\epsilon$ from the density between $V_i$ and $V_j$).

In a sparse graph (such as a dense subgraph of a sparse random graph) we
need a stronger definition of regularity than the one used in dense
graphs. Let $U$ and $V$ be two disjoint subsets of $V(G)$. We say that
they form an $(\epsilon,p)$-regular pair if for any $U'\subseteq U,
V'\subseteq V$ such that $|U'|\geq \epsilon |U|$ and $|V'|\geq \epsilon
|V|$:

$$ |d(U',V')-d(U,V)|\leq \epsilon p{,}$$ {w}here
$d(X,Y)=\frac{|E(X,Y)|}{|X| |Y|}$ is the edge density between two
disjoint sets $X,Y\subseteq {V(G)}$.

Furthermore, an $(\epsilon,p)$-partition of the vertex set of a graph
$G$ is an equipartition of  $V(G)$ into $t$ pairwise disjoint sets
$V(G)=V_1 \cup ... \cup V_t$ in which all but at most $\epsilon t^2$
pairs of sets are $(\epsilon,p)$-regular. For a dense graph{,}
Szemer\'edi's regularity lemma assures us that we can always find a
regular partition of the graph into at most $t(\epsilon)$ parts, but
this is not enough for sparse graphs. For the case of subgraphs of
random graphs{,} one can use a variation by Kohayakawa and R\"odl
\cite{KR} (see also \cite{S}, \cite{ACHKRS} and \cite{L} for some
related results).

In this regularity lemma we add an extra condition. We say that a graph
$G$ on $n$ vertices is $(\eta, p, D)$\textit{-upper-uniform} if for all
disjoint sets $U_1,U_2 \subset V(G)$ such that $|U_i|>\eta n$ one has
$d(U_1,U_2)\leq Dp$. Given this definition we can now state the needed
lemma:

\begin{theorem}[\cite{KR}] \label{thm:randReg} For every $\epsilon>0$,
	$t_0>0$ and $D>0$, there are $\eta,T$ and $N_0$ such that for any
	$p\in[0,1]${, each} $(\eta,p,D)$-upper-uniform graph on $n>N_0$ vertices
	has an $(\epsilon,p)$-regular partition into $t \in [t_0, T]$ parts.
\end{theorem}

In order to estimate the number of copies of a certain graph after
finding a regular partition one needs counting lemmas. We use a
proposition from \cite{CGSS} to show that a certain \color{black}
cluster \color{black} graph is $H$-free, and to give a direct estimate
on the number of copies of $K_m$. To state the proposition we need to
introduce some notation. For a graph $H$ on $k$ vertices, $\{1,...,k\}$,
and for a sequence of integers $\textbf{m}=(m_{ij})_{ij\in E(H)}${,}  we
denote by $\mathcal{G}(H,n',\textbf{m},\epsilon,p)$ the following family
of graphs. The vertex set of each graph in the family is a disjoint
union of sets $V_1,...,V_k$ such that $|V_i|=n'$ for {all} $i$. As for
the edges, for each $ij\in E(H)$ there is an $(\epsilon,p)$-regular
bipartite graph with $m_{ij}$ edges between the sets $V_i$ and $V_j$,
and these are all the edges in the graph. For any $G\in
\mathcal{G}(H,n',\textbf{m},\epsilon,p)$ denote by $G(H)$ the number of
copies of $H$ in $G$ in which every vertex $i$ is in the set $V_i$.

\begin{proposition}[\cite{CGSS}] \label{prop:CountingH} For every graph
	$H$ and every $\delta,d > 0$, there exist{s}  $\xi > 0$ with the
	following property. For every $\eta>0$, there is a $C > 0$ such that if
	$p \geq  Cn^{-1/m_2(H)}$ then w.h.p. the following holds in  $G(n,p)$.
	\begin{enumerate} \item \label{prop:NumOfUnbH} For every $n'\geq \eta
		n$, \textbf{m} with $m_{ij}\geq dp(n')^2$ for all $ij\in H$ and every
		subgraph $G$ of $G(n,p)$ in $\mathcal{G}(H,n',\textbf{m},\epsilon,p)${,}
		\begin{equation} G(H)\geq \xi \left(  \prod_{ij\in E(H)}
		\frac{m_{ij}}{(n')^2} \right) (n')^{v(H)}{.} \end{equation} \item
		\label{prop:NumOfBalH} Moreover, if $H$ is strictly balanced, i.e. for
		every proper subgraph $H'$ of $H$ one has $m_2(H)>m_2(H')${,} then
		\begin{equation} G(H)=(1\pm \delta) \left(  \prod_{ij\in E(H)}
		\frac{m_{ij}}{(n')^2} \right) (n')^{v(H)}{.} \end{equation}
	\end{enumerate} \end{proposition}
	
	Note that the first part tells us that if $G$ is a subgraph of $G(n,p)$
	in $\mathcal{G}(H,n',\textbf{m},\epsilon,p)${,} then it { contains} at
	least one copy of $H$ with vertex $i$ in $V_i$.
	
	\color{black} We can now proceed to the proof of Lemma
	\ref{lem:BigValuesOfp}, starting with a sketch of the argument.
	\color{black} Note that the same steps can be applied to determine
	$ex(G(n,p),T,H)$ for graphs $T$ and $H$ for which
	$ex(n,T,H)=\Theta(n^{v(T)})$ and $p\gg
	\max\{n^{-1/m_2(H)},n^{-1/m_2(T)}\}$.
	
	Let $G$ be an $H$-free subgraph of $G(n,p)$ maximizing the number of
	copies of $K_m$. First apply the sparse regularity lemma (Theorem
	\ref{thm:randReg}) to $G$ and observe using Chernoff and properties of
	the regular partition that there are only a few edges inside clusters
	and between sparse or irregular pairs. By lemma \ref{lem:CopiesOnEdge}
	these edges  do not contribute significantly to the count of $K_m$. We
	can thus consider only graphs $G$ which do not have such edges.
	
	By Proposition \ref{prop:CountingH} the cluster graph must be $H$-free
	and taking $G$ to be maximal we can assume all pairs in the cluster
	graph have the maximal possible density. Applying Proposition
	\ref{prop:CountingH} again to count the number of copies of $K_m$
	reduces the problem to the dense case solved in \cite{ASh}.
	
	We continue with the full details of the proof. \color{black}

	\begin{proof}[Proof of Lemma \ref{lem:BigValuesOfp}]

		{ A $(k-1)$-partite graph with sides of size $\frac{n}{k-1}$
			\color{black} each \color{black} is an $n$-vertex $H$-free graph
			containing} $(1+o(1))\binom{k-1}{m} (\frac{n}{k-1})^m$ copies of
		$K_m$. 
		{W}e can get a random subgraph of it by keeping each edge with
		probability $p$, independently of the other edges. 
		Then by Lemma \ref{lem:CopiesOfK_mInSided} the number of copies of
		$K_m$ in it is $(1+o(1))\binom{k-1}{m}(\frac{n}{k-1})^m
		p^{\binom{m}{2}}$ \color{black} w.h.p.,  \color{black} proving the
		required lower bound on $ex(G(n,p),K_m,H)$.
		


		For the upper bound we need to show that {no} $H$-free subgraph of
		$G(n,p)$ { has} more than $(1+o(1))\binom{k-1}{m}(\frac{n}{k-1})^m
		p^{\binom{m}{2}}$ copies of $K_m$.  Let $G$ be an $H$-free subgraph of
		$G(n,p)$ with the maximum number copies of $K_m$. To use Theorem
		\ref{thm:randReg}{,} we need to show that $G$ is
		$(\eta,p,D)$-upper-uniform for some constant {$D$}, say $D=2$, and
		$\eta>0$. Indeed, taking any two disjoint subsets $V_1,V_2$ of size
		$\geq \eta n${,} we get that the number of edges between them is
		bounded by the number of edges between them in $G(n,p)$, which is
		distributed like $Bin(|V_1|\cdot |V_2|,p)$. Applying {P}art
		\ref{lem:ChernoffUppeBigA} of Lemma \ref{lem:Chernoff} and the union
		bound gives us that w.h.p.\ the number of edges between any two such
		sets is $\leq 2|V_1|\cdot |V_2|p$ and so $d(V_1,V_2)<2p$ as needed.
		Thus by Theorem \ref{thm:randReg}{,} $G$ admits an
		$(\epsilon,p)$-regular partition into $t$ parts $V(G)=V_1 \cup \dots
		\cup V_t$.
		
		Define the {\em cluster graph of $G$} to be the graph whose vertices
		are the sets $V_i$ of the partition and there is an edge between two
		sets if the density of the bipartite graph induced by them is at least
		$\delta p$ for some fixed small $\delta>0$\color{black}, \color{black}
		and they form an $(\epsilon,p)$-regular pair.
		
		First we show that w.h.p.\ the cluster graph is $H$-free. Assume that
		there is a copy of $H$ in the cluster graph, { induced} by the sets
		$V_1,\dots,V_{v(H)}$. \color{black}Consider these \color{black} sets
		in the original graph $G$. To apply {P}art \ref{prop:NumOfUnbH} of
		Proposition \ref{prop:CountingH} first note that indeed $p\geq
		Cn^{-1/m_2(H)}$. Furthermore if $ij\in E(H)$ then by the definition of
		the cluster graph $V_i$ and $V_j$ form an $(\epsilon,p)$-regular pair
		and there are at least $\delta p(\frac{n}{t})^2$ edges between them.
		Thus the graph spanned by the {edges} between $V_1,\dots,V_{v(H)}$ in
		$G$ is in $\mathcal{G}(H,\frac{n}{t},\textbf{m},\epsilon,p)$ where
		$m_{ij}\geq \delta p(\frac{n}{t})^2$, and so w.h.p.\ it contains a
		copy of $H$ with vertex $i$ in the set $V_i$. This contradicts the
		fact that $G$ was $H$-free to start with.
		
		If the cluster graph is indeed $H$-free, as proven in \cite{ASh},
		Proposition 2.2,  since $\chi(H)>m$ then
		$ex(t,K_m,H)=(1+o(1))\binom{k-1}{m}(\frac{t}{k-1})^m$. This gives a
		bound on the number of copies of $K_m$ in the cluster graph. For sets
		$V_1,...,V_m$ that span a copy of $K_m$ in the cluster graph we would
		like to bound the number of copies of $K_m$ with a vertex in each set
		in the original graph $G$.
		
		To do this{,} we use {P}art \ref{prop:NumOfBalH} of Proposition
		\ref{prop:CountingH}. Note that we cannot use Lemma
		\ref{lem:CopiesOfK_mInSided} as we need it for every subgraph of
		$G(n,p)$ and not only for a specific one. Part \ref{prop:NumOfBalH}
		can be applied only to balanced graphs, and indeed any subgraph of
		$K_m$ is $K_{m'}$ for some $m'<m$ and
		$m_2(K_{m'})=\frac{m'+1}{2}<\frac{m+1}{2}=m_2(K_m)$. As we would like
		to have a upper bound on the number of copies of $K_m$ with a vertex
		in each set, we can assume that the bipartite graph between $V_i$ and
		$V_j$ has all of the edges from $G(n,p)$.
		
		\color{black}By  {P}arts \ref{lem:ChernoffLower} and
		\ref{lem:ChernoffUppe} of Lemma \ref{lem:Chernoff}, \color{black}
		w.h.p.\ for any $V_i$ and $V_j$ of size $\frac{n}{t}$,
		$|E(V_i,V_j)|=(1+o(1))p(\frac{n}{t})^2$. Thus the graph { induced} by
		the sets $V_1,...,V_m$ in $G(n,p)$ is in
		$\mathcal{G}(K_m,\frac{n}{t},\textbf{m},\epsilon,p)$ where $m_{ij}=
		(1+o(1))p(\frac{n}{t})^2$ for any pair $ij$. From this the number of
		copies of $K_m$ in $G$ with a vertex in every $V_i$ is at most
		$\color{black}(1+o(1))\color{black} p^{\binom{m}{2}} (\frac{n}{t})^{m}
		$. Plugging this into the bound on the number of copies of $K_m$ in
		the cluster graph \color{black} implies \color{black} that the number
		of copies of $K_m$ coming from \color{black} copies \color{black} of
		$K_m$ in the cluster is w.h.p.\ at most
		
		%

		
		$$ (1+o(1))\binom{k-1}{m} (\frac{t}{k-1})^m \cdot
		p^{\binom{m}{2}}(\frac{n}{t})^m=(1+o(1))\binom{k-1}{m}
		(\frac{n}{k-1})^m \cdot p^{\binom{m}{2}}{.}$$
		
		It is left to show that the number of copies of $K_m$ coming from
		other parts of the graph is negligible.

		To do this we show that the number of edges inside clusters and
		between non-dense or irregular pairs is negligible. By Chernoff (Part
		\ref{lem:ChernoffUppeBigA} of Lemma \ref{lem:Chernoff}) the number of
		edges inside a cluster is at most $2p\binom{n/t}{2} t\leq
		2p\frac{n^2}{t}$. The number of irregular pairs is at most $\epsilon
		t^2$, and again by Chernoff there are no more than
		$2p(\frac{n}{t})^2\cdot \epsilon t^2 =2\epsilon pn^2$ edges between
		\color{black} these \color{black} pairs. Finally, the number of edges
		between non-dense pairs is at most $\delta p (\frac{n}{t})^2
		t^2=\delta pn^2$.
		
		As $\epsilon$, $\delta$ and $\frac{1}{t}$ can be chosen as small as
		needed we get that  the number of such edges is $o(n^2p)$ Thus we may
		apply Lemma \ref{lem:CopiesOnEdge} and conclude that the number of
		copies of $K_m$ \color{black}containing at least one of these edges is
		\color{black}  $o(n^m p^{\binom{m}{2}})$.

		\color{black}Therefore, \color{black} 
		for any $H$-free $G\subset G(n,p)$ the number of copies of $K_m$ in
		$G$ is at most $ (1+o(1))\binom{k-1}{m} (\frac{n}{k-1})^m
		p^{\binom{m}{2}}$ as needed. \end{proof}
	

	
	The proofs of the other two lemmas are a bit simpler.
	

	\begin{proof}[Proof of Lemma \ref{lem:SmallValuesOfp}] \color{black} As
		$p<n^{-1/m_2(K_m)-\delta}$ we can first delete all copies of $K_m$
		sharing an edge \color{black} with other copies \color{black} and by
		Lemma \ref{lem:oneK_mPerEdge} we deleted w.h.p.\ only
		$o(n^mp^{\binom{m}{2}})$ copies of $K_m$. 
		Let $H'$ be a subgraph of $H$ for which
		$\frac{e(H')-1}{v(H')-1}=m_2(H)$. Let $e$ be an edge of $H'$ and
		define $\{H_i\}$ to be the family of \color{black} all \color{black}
		graphs obtained by gluing a copy of $K_m$ to the edge $e$ in $H'$ and
		allowing any further intersection. Note that the number of graphs in
		$\{H_i\}$ depends only on $H'$ and $m$. One can make $G$ into an
		$H$-free graph by deleting the edge $e$ from every copy of a graph
		from  $\{H_i\}$ and every edge that does not take part in a copy of
		$K_m$. As we may assume every edge takes part in at most one copy of
		$K_m$ it is enough to show that the number of copies of graphs from
		$\{H_i\}$ is $o(n^m p^{\binom{m}{2}})$.
		
		For a fixed graph $J$, let $X_J$ be the random variable counting the
		number of copies of $J$ in $G\sim G(n,p)$. With this notation{,}
		\begin{align*} \E(X_{K_m})&=\Theta(n^{m}p^{\binom{m}{2}})=\Theta(n^2 p
		(np^{m_2(K_m)})^{m-2}){,}\\ \E(X_{H_i})&=\Theta (n^2p
		(np^{m_2(H_i)})^{v(H_i)-2}){.} \end{align*}
		
		As $m_2(H_i)\geq m_2(K_m)$ and $p<n^{-1/m_2(K_m)}${,} we get
		$np^{m_2(H_i)}\leq np^{m_2(K_m)}\color{black}\ll\color{black}1$.
		Furthermore as $v(H_i)>m$ (otherwise $H'$ would be a subgraph of
		$K_m$) we get that
		$(np^{m_2(H_i)})^{v(H_i)-2}=o((np^{m_2(K_m)})^{m-2})$ and thus
		$\mathbb{E}(X_{H_i})=o(\mathbb{E}(X_{K_m})$.
		
		If $p$ is such that the expected number of copies of $K_m$, the graphs
		$\{H_i\}$ and any of their subgraphs goes to infinity as $n$ goes to
		infinity we can apply Theorem \ref{lem:ExpectedNumLemma} and get that
		$X_{K_m}=(1+o(1))\binom{n}{m}p^{\binom{m}{2}}$ and the number of
		copies of $H_i$ is \color{black}w.h.p.\ \color{black}
		$(1+o(1))\E(X_{H_i})=o(\E (X_{K_m}))$. Thus if we remove all edges
		playing the part of $e$ in any $H_i$ the number of copies of $K_m$
		will still be $(1+o(1))\binom{n}{m}p^{\binom{m}{2}}$.

		Finally, if the number of copies of some subgraph of $H_i$ does not
		tend to infinity as $n$ tends to infinity we can remove all of the
		edges taking part in it, and the number of edges removed is
		$o(\binom{n}{m}p^{\binom{m}{2}})$. As each edge takes part in a single
		copy of $K_m$, we still get that the number of copies of $K_m$ in this
		graph is $(1+o(1))\binom{n}{m}p^{\binom{m}{2}}$\color{black},
		\color{black} as needed. \end{proof} \color{black} 
	
	\color{black} 
	\begin{proof}[Proof of Lemma \ref{lem:MidRangeForEasyOrd}] Let
		$n^{-1/m_2(K_m)-\epsilon}<p\ll n^{-1/m_2(H)}$ and $G\sim G(n,p)$. Let
		$H'$ be a subgraph of $H$ for which $\frac{e(H')-1}{v(H')-2}=m_2(H)$.
		We show that if $G$ is made $H$-free by removing a single edge from
		every copy of $H'$ then the number of copies of $K_m$ deleted is
		$o(\binom{n}{m}p^{\binom{m}{2}})$. Theorem \ref{lem:ExpectedNumLemma}
		assures us that the number of copies of $K_m$ in $G$ is
		$(1+o(1))\binom{n}{m}p^{\binom{m}{2}}$ and so it stays essentially the
		same after removing all copies of $H'$.

		The expected number of copies of $H'$ in $G$ is
		$$\E[\mathcal{N}(G,H')]=\Theta (n^2p
		(np^{m_2(H')})^{v(H')-2})=o(n^2p).$$ Thus by Markov's inequality
		w.h.p.\ $\mathcal{N}(G,H')=o(n^2p)$. If $p\gg n^{-1/m_2(K_m)}$ then by
		Lemma \ref{lem:CopiesOnEdge} deleting all these edges removes only
		$o(n^m p^{\binom{m}{2}})$ copies of $K_m$.
		
		As for smaller values of $p$, namely $p\leq O(n^{-1/m_2(K_m)})$, it
		follows that $$\E[\mathcal{N}(G,H')]=\Theta (n^2p
		(np^{m_2(H')})^{v(H')-2})\leq n^{-\beta}n^2p$$ for some $\beta>0$. By
		Markov's inequality w.h.p.\ the number of edges taking part in a copy
		of $H'$ in $G$ is at most $n^{-\alpha}n^2p$ for, say,
		$\alpha=\frac{\beta}{2}$.
		
		Since $p\leq O(n^{-1/m_2(K_m)})$ Lemma \ref{lem:CopiesOnEdge} cannot
		be applied directly. To take care of this, define $q=n^{2\epsilon}p\gg
		n^{-1/m_2(K_m)}$. Lemma \ref{lem:CopiesOnEdge} applied to $G(n,q)$
		implies that a set of at most $n^{-\alpha}qn^2$ edges takes part in no
		more than $n^{-\delta}n^m q^{\binom{m}{2}}$ copies of $K_m$ , where
		$\delta=\delta(\alpha)>0$.
		
		The number of copies of $K_m$ containing a member of a set of edges in
		$G(n,p)$ is monotone in $p$ and in the size of the set. Thus when
		deleting a single edge from each copy of $H'$ in $G(n,p)$ the number
		of copies of $K_m$ removed is w.h.p.\ at most
		$n^{-\delta}n^mq^{\binom{m}{2}}=n^{-\delta+2\binom{m}{2}\epsilon}n^mp^{\binom{m}{2}}$. Choosing $\epsilon$ small enough implies that the number of copies of $K_m$ removed is $o(n^m p^{\binom{m}{2}})$ as needed. \end{proof} \color{black}

	\section{Construction of graphs with small 2-density}
	
	In the proof of Theorem \ref{thm:existOfH} we construct a family of
	graphs $\{G(k,\epsilon)\}$ that are $k$-critical and
	$m_2(G(k,\epsilon))=(1+\epsilon)M_k$ where $M_k$ is the smallest
	possible value of $m_2$ for a $k$-chromatic graph. 
	{The following notation will be useful.} For a graph $G$ and
	$A\subseteq V(G)$ such that $|A|\geq 3$, let
	$d^{(2)}_G(A)=\frac{e(G[A])-1}{|A|-2}$. 
	{By definition, } $m_2(G)=\min_{A\subseteq V(G)\,:\, |A|\geq
		3}d^{(2)}_G(A)$.

	\begin{proof}[Proof of Theorem \ref{thm:existOfH}]
		
		We construct the graphs $G(k,\epsilon)$ in three steps. In Step 1 we
		construct so called $(k,t)$-{\em towers} and derive some useful
		properties of them. In Step 2 we make from $(k,t)$-{\em towers} more
		complicated $(k,t)$-{\em complexes} and {\em supercomplexes}, and in
		Step 3 we replace each edge in a copy of $K_k$ with a supercomplex and
		prove the needed.

		\subsubsection*{Step 1: Towers} Let $t=t(\epsilon)=\lceil
		k^3/\epsilon\rceil$. The $(k,t)$-{\em tower with base
			$\{v_{0,0}v_{0,1}\}$} is the graph $T_{k,t}$ defined as follows. The
		vertex set of $T_{k,t}$ is $V_0\cup V_1\cup\ldots \cup V_{t}$, where
		$V_0=\{v_{0,0},v_{0,1}\}$ and for $1\leq i\leq t$,
		$V_i=\{v_{i,0},v_{i,1},\ldots,v_{i,k-2}\}$. For $i=1,\ldots,t$,
		$T_{k,t}[V_i]$ induces $K_{k-1}-e$ with the missing edge
		$v_{i,0}v_{i,1}$. Also for $i=1,\ldots,t$, vertex $v_{i-1,0}$ is
		adjacent to $v_{i,j}$ for all  $0\leq j\leq (k-2)/2$ and vertex
		$v_{i-1,1}$ is adjacent to $v_{i,j}$ for all  $(k-1)/2\leq j\leq k-2$.
		There are no other edges.

		\begin{figure} [H] \centering

			\begin{tikzpicture}[scale=1, transform shape] \tikzstyle{every node}
			= [circle, minimum width=8pt, inner sep=0pt] \node (v00) at (1.5, 0)
			{$v_{0,0}$}; \node (v01) at +(2.5, 0) {$v_{0,1}$};
			
			\node (v10) at +(1,1) {$v_{1,0}$}; \node (v11) at +(2,1 )
			{$v_{1,1}$}; \node (v12) at +(3,1 ) {$v_{1,2}$};
			
			\node (v20) at +(1,2) {$v_{2,0}$}; \node (v21) at +(2,2 )
			{$v_{2,1}$}; \node (v22) at +(3,2 ) {$v_{2,2}$};

			\node (v30) at +(1,3) {}; \node (v31) at +(2,3 ) {}; \node (v32) at
			+(3,3 ) {};
			
			\foreach \from/\to in { v00/v10, v00/v11} \draw [-] (\from) edge
			(\to); \path[-] (v01) edge (v12); \foreach \from/\to in { v11/v12}
			\draw [-] (\from) -- (\to); \path[-] (v12)    edge[bend right=30]
			(v10);

			\foreach \from/\to in { v10/v20, v10/v21} \draw [-] (\from) -- (\to);
			\path[-] (v11) edge (v22);
			
			\foreach \from/\to in { v21/v22} \draw [-] (\from) -- (\to); \path[-]
			(v22)    edge[bend right=30] (v20);
			
			\foreach \from/\to in { v20/v30, v20/v31} \draw [-] (\from) edge
			(\to); \path[-] (v21) edge (v32);

			\draw (2,3) circle (0.02cm); \fill (2,3) circle (0.02cm);
			
			\draw (2,3.5) circle (0.02cm); \fill (2,3.5) circle (0.02cm);
			
			\draw (2,4) circle (0.02cm); \fill (2,4) circle (0.02cm);
			
			\node (v40) at +(1,4) {}; \node (v41) at +(2,4 ) {}; \node (v42) at
			+(3,4 ) {};
			
			\node (v50) at +(1,5) {$v_{t,0}$ }; \node (v51) at +(2,5 )
			{$v_{t,1}$}; \node (v52) at +(3,5 ) {$v_{t,2}$};

			\foreach \from/\to in { v40/v50, v40/v51} \draw [-] (\from) -- (\to);
			\path[-] (v41) edge (v52);
			
			\foreach \from/\to in { v51/v52} \draw [-] (\from) -- (\to); \path[-]
			(v52)    edge[bend right=30] (v50);

			\end{tikzpicture}
			
			\caption{$T_{4,t}$} \label{fig:M1} \end{figure}

		\medskip By construction,
		$|E(T_{k,t})|=t\left(\binom{k-1}{2}-1+(k-1)\right)=t\frac{(k+1)(k-2)}{2}=(|V(T_{k,t})|-2)\frac{(k+1)(k-2)}{2(k-1)}$, that is, \begin{equation}\label{equ1} d^{(2)}_{T_{k,t}}(V(T_{k,t}))=\frac{(k+1)(k-2)}{2(k-1)}-\frac{1}{|V(T_{k,t})|-2}. \end{equation} Also, since for each $i=1,\ldots,t$, $|N(v_{i-1,1})\cap V_i|\leq (k-1)/2$ and among the $\lceil (k-1)/2\rceil$ neighbors of $v_{i-1,0}$ in $V_i$, $v_{i,0}$ and $v_{i,1}$ are not adjacent to each other, \begin{equation}\label{equ2} \omega(T_{k,t})=k-2. \end{equation}

		Our first goal  is to show that $T_{k,t}$ has no dense subgraphs. We
		will use the language of potentials to prove this. For a graph $H$ and
		$A\subseteq V(H)$, let $$\mbox{
			$\rho_{k,H}(A)=(k+1)(k-2)|A|-2(k-1)|E(H[A])|$ be the {\em potential of
				$A$ in $H$}.}$$ A convenient property of potentials is that if
		$|A|\geq 3$, then \begin{equation}\label{equ10}
		\mbox{$\rho_{k,H}(A)\geq 2(k+1)(k-2)-2(k-1)$ if and only if
			$d^{(2)}_{H}(A)\leq \frac{(k+1)(k-2)}{2(k-1)}$,} \end{equation} but
		potentials are also well defined for sets with cardinality  two or
		less.
		
		\begin{lemma}\label{cl1} Let $T=T_{k,t}$. For every $A\subseteq V(T)$,
			
			\begin{equation}\label{equ3} \mbox{if $|A|\geq 2$, then
				$\rho_{k,T}(A)\geq 2(k+1)(k-2)-2(k-1)$.} \end{equation} Moreover,
			\begin{equation}\label{equ4} \mbox{if $V_0\subseteq A$, then
				$\rho_{k,T}(A)\geq 2(k+1)(k-2)$.} \end{equation} \end{lemma}
		
		{\bf Proof.}  Suppose  the lemma is not true. Among $A\subseteq V(T)$
		with $|A|\geq 2$ for which~\eqref{equ3} or~\eqref{equ4} does not hold,
		choose $A_0$ 
		with the smallest size. Let $a=|A_0|$.
		
		If $a=2$, then $\rho_{k,T}(A_0)=2(k+1)(k-2)-2(k-1)|E(T[A_0])|\geq
		2(k+1)(k-2) -2(k-1)$. Moreover, if $a=2$ and $V_0\subseteq A$, then
		$V_0=A_0$ and so $E(T[A_0])=\emptyset$. This contradicts the choice of
		$A_0$. So \begin{equation}\label{equ5} a\geq 3. \end{equation}

		Let $i_0$ be the maximum $i$ such that $A_0\cap V_i\neq\emptyset$.
		By~\eqref{equ5}, $i_0\geq 1$. Let $A'=A_0\cap V_{i_0}$ and $a'=|A'|$.
		
		\smallskip {\bf Case 1:} $a'\leq k-2$ and $a-a'\geq 2$. Since
		$|(A_0-A')\cap V_0|=|A_0\cap V_0|$, by the minimality of
		$a$,~\eqref{equ3} and~\eqref{equ4} hold for $A_0-A'$. Thus,
		$$\rho_{k,T}(A_0)\geq
		\rho_{k,T}(A_0-A')+a'(k+1)(k-2)-2(k-1)\left(a'+\binom{a'}{2}\right) $$
		$$=\rho_{k,T}(A_0-A')+a'\left[(k^2-k-2)-2k+2-(k-1)(a'-1)\right]. $$
		Since $k\geq 4$ and $a'\leq k-2$, the expression in the brackets is at
		least $k^2-3k-(k-1)(k-3)=k-3>0$, contradicting the choice of $A_0$.

		\smallskip {\bf Case 2:} $A'=V_{i_0}$  and ${a-a'}\geq 2$. Then
		$a'=k-1$. As in Case 1,~\eqref{equ3} and~\eqref{equ4} hold for
		$A_0-A'$. Thus, $$\rho_{k,T}(A_0)\geq
		\rho_{k,T}(A_0-A')+a'(k+1)(k-2)-2(k-1)\left(a'+\binom{a'}{2}-1\right)
		$$ $$=\rho_{k,T}(A_0-A')+(k-1)\left[(k^2-k-2)-2(k-1)-(k-1)((k-1)-1)+2\right] $$ $$\geq \rho_{k,T}(A_0-A')+(k-1)^2\left[(k-2)-(k-2)\right]=\rho_{k,T}(A_0-A'), $$ contradicting the minimality of $A_0$.
		
		\smallskip {\bf Case 3:} {$a=a'$, i.e.,} $A_0=A'$. Then
		$V_0\not\subseteq A_0$ and $ a'\geq 3$. If $a\leq k-2$, then
		\begin{equation}\label{equ7} \rho_{k,T}(A_0)\geq
		a(k+1)(k-2)-2(k-1)\binom{a}{2} =a[(k+1)(k-2)-(k-1)(a-1)].
		\end{equation} Since the RHS of~\eqref{equ7} is quadratic in $a$ with
		the negative leading coefficient, it is enough to evaluate the RHS
		of~\eqref{equ7} for $a=2$ and $a=k-2$. For $a=2$, it is
		$2(k+1)(k-2)-2(k-1)$, exactly as in~\eqref{equ3}. For $a=k-2$, it is
		$$(k-2)[(k+1)(k-2)-(k-1)(k-2-1)]=(3k-5)(k-2), $$ and $(3k-5)(k-2)\geq
		2(k+1)(k-2)-2(k-1)$ for $k\geq 4$. If $a=k-1$, then $A_0=V_i$ and
		$$\rho_{k,T}(A_0)=
		a(k+1)(k-2)-2(k-1)\left(\binom{a}{2}-1\right)=(k-1)((k+1)(k-2)-(k-1)(k-2)+2) $$ $$=2(k-1)^2>2(k+1)(k-2)-2(k-1). $$
		
		\smallskip {\bf Case 4:} ${a-a'}= 1$. As in Case 3, $V_0\not\subseteq
		A_0$ and $ a'\geq 2$. Let $\{z\}=A_0-A'$. Repeating the argument of
		Case 3, we obtain that $\rho_{k,T}(A')\geq 2(k+1)(k-2)-2(k-1)$. So, if
		$d_{T[A_0]}(z)\leq \frac{k-1}{2}$, then $$\rho_{k,T}(A_0)\geq
		\rho_{k,T}(A')+(k+1)(k-2)-2(k-1)\frac{k-1}{2}=\rho_{k,T}(A')+k-3>\rho_{k,T}(A'), $$ a contradiction to the choice of $A_0$. And the only way that $d_{T[A_0]}(z)> \frac{k-1}{2}$, is that $z=v_{i-1,0}$, $k$ is even, and $A'\supseteq \{v_{i,0},\ldots,v_{i,(k-2)/2}\}$. Then edge $v_{i,0}v_{i,1}$ is missing in $T[A']$, and hence \begin{equation}\label{equ8} \rho_{k,T}(A_0)=(a'+1)(k+1)(k-2)-2(k-1)\left(\binom{a'}{2}-1+k/2\right). \end{equation} Since the RHS of~\eqref{equ8} is quadratic in $a'$ with the negative leading coefficient and $a'\geq k/2$, it is enough to evaluate the RHS of~\eqref{equ8} for $a'=k/2$ and $a'=k-1$. For $a'=k/2$, it is $$\frac{k+2}{2}(k+1)(k-2)-(k-1)\left(\frac{k(k-2)}{4}-2+k\right)=\frac{k-2}{4}(k^2+3k+8). $$ Since $\frac{k^2+3k+8}{4}>2k$ for $k\geq 4$ and $2k(k-2)>2(k+1)(k-2)-2(k-1)$, we satisfy~\eqref{equ3}. If $a'=k-1$, then the RHS of~\eqref{equ8} is $$ k(k+1)(k-2)-(k-1)[(k-1)(k-2)-2+k]=(k-2)[k(k+1)-(k-1)^2-k+1]$$ $$=2k(k-2)>2(k+1)(k-2)-2(k-1).\qed$$
		
		{Graph} $T_{k,t}$ also has good coloring properties.
		
		\begin{lemma}\label{lem2} Suppose $T_{k,t}$ has a $(k-1)$-coloring $f$
			such that \begin{equation}\label{e3} f(v_{0,1})=f(v_{0,0}).
			\end{equation} Then for every $1\leq i\leq t$,
			\begin{equation}\label{e4} f(v_{i,1})=f(v_{i,0}). \end{equation}
		\end{lemma}
		
		{\bf Proof.} We prove~\eqref{e4} by induction on $i$. For $i=0$, this
		is~\eqref{e3}. Suppose~\eqref{e4} holds for $i=j<t$. Since
		$V_{j+1}\subseteq N(v_{i,0})\cup N(v_{i,1})$, the color
		$f(v_{j,1})=f(v_{j,2})$ is not used on $V_{j+1}$ and thus
		$f(v_{j+1,1})=f(v_{j+1,0})$, as claimed. \qed

		\subsubsection*{Step 2: Tower complexes} A {\em tower complex}
		$C_{k,t}$ is the union of $k$ copies $T^{1}_{k,t},\ldots,T^{k}_{k,t}$
		of the tower $T_{k,t}$ such that every two of them have the common
		base $V_0^1=\ldots=V^k$, are vertex-disjoint apart from that, and have
		no edges between $T^i_{k,t}-V_0^i$ and $T^j_{k,t}-V_0^j$ for $j\neq
		i$. This common base $V^0=\{v_{0,0},v_{0,1}\}$ will be called {\em the
			base of $C_{k,t}$}.
		
		Lemma~\ref{cl1} naturally extends to complexes as follows.
		
		
		\begin{lemma}\label{cl2} Let $C=C_{k,t}$. For every $A\subseteq V(C)$,
			\begin{equation}\label{equ3'} \mbox{if $|A|\geq 2$, then
				$\rho_{k,C}(A)\geq 2(k+1)(k-2)-2(k-1)$.} \end{equation} Moreover,
			\begin{equation}\label{equ4'} \mbox{if $A\supseteq V_0$, then
				$\rho_{k,C}(A)\geq 2(k+1)(k-2)$.} \end{equation} \end{lemma}

		{\bf Proof.} Let $A\subseteq V(C)$ with $|A|\geq 2$, and $A_0=A\cap
		V_0$. Let $A_i=A\cap V(T^{i}_{k,t})$ if $A\cap V(T^{i}_{k,t})-V_0\neq
		\emptyset$, and $A_i=\emptyset$ otherwise. Let $I=\{i\in
		[t]\,:\,A_i\neq\emptyset\}$. If $|I|\leq 1$, then $A$ is a subset of
		one of the towers, and we are done by Lemma~\ref{cl1}. So let $|I|\geq
		2$.
		
		{\bf Case 1:} $V_0\subseteq A$. Then for each nonempty $A_i$,
		$|A_i|\geq 3$ and by Lemma~\ref{cl1}, $\rho_{k,C}(A_i)\geq
		2(k+1)(k-2)$. So, by the definition of the potential,
		$$\rho_{k,C}(A)=\sum_{i\in I}\rho_{k,C}(A_i)-(|I|-1)2(k+1)(k-2)\geq
		|I|2(k+1)(k-2)-(|I|-1)2(k+1)(k-2)=2(k+1)(k-2).$$
		
		{\bf Case 2:} $V_0\cap A=\{v_{0,j}\}$, where $j\in \{1,2\}$. Then for
		each nonempty $A_i$, $|A_i|\geq 2$ and by Lemma~\ref{cl1},
		$\rho_{k,C}(A_i)\geq 2(k+1)(k-2)-2(k-1)$. So, by the definition of the
		potential and the fact that $|I|\geq 2$, $$\rho_{k,C}(A)=\sum_{i\in
			I}\rho_{k,C}(A_i)-(k+1)(k-2)(|I|-1)$$ $$\geq
		|I|(2(k+1)(k-2)-2(k-1))-(k+1)(k-2)(|I|-1)
		=|I|((k+1)(k-2)-2(k-1))+(k+1)(k-2)$$ $$\geq
		2((k+1)(k-2)-2(k-1))+(k+1)(k-2)>2(k+1)(k-2)-2(k-1),$$ when $k\geq 4$.
		
		{\bf Case 3:} $V_0\cap A=\emptyset$. Then $\rho_{k,C}(A)=\sum_{i\in
			I}\rho_{k,C}(A_i)$. Since $\rho_{k,C}(A_i)\geq (k+1)(k-2)$ for every
		$i\in I$ and $|I|\geq 2$, $\rho_{k,C}(A)\geq 2(k+1)(k-2)$, as
		claimed.\qed
		
		\medskip Given a tower complex $C_{k,t}$, let
		$W_0=\{v^{1}_{t,0},\ldots,v^{k}_{t,0}\}$ and
		$W_1=\{v^{1}_{t,1},\ldots,v^{k}_{t,1}\}$. Then the auxiliary {\em
			bridge graph} $B_{k,t}$ is the bipartite graph with parts $W_0$ and
		$W_1$ whose edges are defined as follows. For each pair $(i,j)$ with
		$1\leq i<j\leq k$, if $j-i\leq k/2$, then $B_{k,t}$ contains edge
		$v^i_{t,0}v^j_{t,1}$, otherwise it contains edge $v^i_{t,1}v^j_{t,0}$.
		There are no other edges.

		\begin{figure}[H] \centering

			\begin{tikzpicture}[scale=1, transform shape] \tikzstyle{every node}
			= [circle, minimum width=8pt, inner sep=0pt] \node(v10) at
			(0,0){$v^1_{t,0}$}; \node(v20) at (0,1){$v^2_{t,0}$}; \node(v30) at
			(0,2){$v^3_{t,0}$}; \node(v40) at (0,3){$v^4_{t,0}$};

			\node(v11) at (2,0){$v^1_{t,1}$}; \node(v21) at (2,1){$v^2_{t,1}$};
			\node(v31) at (2,2){$v^3_{t,1}$}; \node(v41) at (2,3){$v^4_{t,1}$};

			\foreach \from/\to in { v10/v21, v10/v31, v20/v31, v20/v41, v30/v41,
				v40/v11} \draw [-] (\from) edge (\to);
			
			\end{tikzpicture}
			
			\caption{$B_{4,t}$} \label{fig:M2} \end{figure}
		
		By construction, $B_{k,t}$ has exactly $\binom{k}{2}$ edges, and the
		maximum degree of $B_{k,t}$ is $ \lfloor k/2\rfloor$. It is important
		that \begin{equation}\label{equ9} \mbox{for each $1\leq i<j\leq k$, an
			edge in $B_{k,t}$ connects $\{v^i_{t,0},v^i_{t,1}\}$ with
			$\{v^j_{t,0},v^j_{t,1}\}$.} \end{equation}

		\medskip The {\em supercomplex} $S_{k,t}$ is obtained from a tower
		complex $C_{k,t}$ by adding to it all edges of $B_{k,t}$. The main
		properties of $S_{k,t}$ are stated in the next three lemmas.

		\begin{lemma}\label{lem2'} For each  $(k-1)$-coloring $f$ of
			$S_{k,t}${,} 
			\begin{equation}\label{eq3} f(v_{0,1})\neq f(v_{0,0}). \end{equation}
		\end{lemma}
		
		{\bf Proof.} Suppose $S_{k,t}$ has a $(k-1)$-coloring $f$ with
		$f(v_{0,1})=f(v_{0,0}).$ Then by Lemma~\ref{lem2},
		$f(v^i_{t,1})=f(v^i_{t,0})$ for every $1\leq i\leq k$. Thus
		by~\eqref{equ9}, the $k$ colors
		$f(v^1_{t,0}),f(v^2_{t,0}),\ldots,f(v^k_{t,0})$ are all distinct, a
		contradiction.\qed

		\begin{lemma}\label{cl3} Let $S=S_{k,t}$ with base $V_0$. For every
			$A\subseteq V(S)-V_0$, \begin{equation}\label{equ3''} \mbox{if
				$|A|\geq 2$, then   $\rho_{k,S}(A)\geq 2(k+1)(k-2)-2(k-1)$.}
			\end{equation} \end{lemma}

		{\bf Proof.}  Suppose  the lemma is not true. Let $C$ be the copy of
		$C_{k,t}$ from which we obtained $S$ by adding the edges of
		$B=B_{k,t}$. Among $A\subseteq V(S)-V_0$ with $|A|\geq 2$ and
		$\rho_{k,S}(A)< 2(k+1)(k-2)-2(k-1)$, choose $A_0$ with the smallest
		size. Let $a=|A_0|$. Let $I=\{i\in [t]\,:\,A_0\cap V(T^{i}_{k,t})
		\neq\emptyset\}$. If $|I|\leq 1$, then $A$ is a subset of one of the
		towers, and we are done by Lemma~\ref{cl1}. So let $|I|\geq 2$.

		If $a=2$, then $$\rho_{k,S}(A_0)= a(k+1)(k-2)-2(k-1)|E(S[A_0])|\geq
		2(k+1)(k-2)-2(k-1),$$  contradicting the choice of $A_0$. So $ a\geq
		3$. Furthermore, if $a=3$, then since $|I|\geq 2$, $B_{k,t}$ is
		bipartite, and $v^i_{t,0}v^i_{t,1}\notin E(S)$ for any $i$, the graph
		$S[A_0]$ has at most two edges and so $\rho_{k,S}(A_0)\geq
		3(k+1)(k-2)-2(2(k-1))>2(k+1)(k-2)-2(k-1)$. Thus
		\begin{equation}\label{equ5'} a\geq 4. \end{equation}
		
		If $d_{S[A_0]}(w)\leq \frac{k-1}{2}$ for some $w\in A_0$, then
		$$\rho_{k,S}(A_0-w)\leq
		\rho_{k,S}(A_0)-(k+1)(k-2)+\frac{k-1}{2}2(k-1)=\rho_{k,S}(A_0)+3-k<\rho_{k,S}(A_0).$$ By~\eqref{equ5'}, this contradicts the minimality of $a$. So, \begin{equation}\label{equ11} \mbox{$\delta(S[A_0])\geq \frac{k}{2}$. In particular, $a\geq 1+ \frac{k}{2}$.} \end{equation}
		
		Let $E(A_0,B)$ denote the set of edges of $B$ both ends of which are
		in $A_0$. Then since $A_0\cap V_0=\emptyset$,
		\begin{equation}\label{equ12}
		\rho_{k,S}(A_0)=\rho_{k,C}(A_0)-2(k-1)|E(A_0,B)|=\sum_{i\in
			I}\rho_{k,C}(A_i)-2(k-1)|E(A_0,B)|. \end{equation}
		
		Let $I_1=\{i\in I\,:\,|A_0\cap V(T^{i}_{k,t})|=1\}$ and $I_2=I-I_1$.
		By Lemma~\ref{cl1}, for each $i\in I_2$, $\rho_{k,S}(A_i)\geq
		2(k+1)(k-2)-2(k-1)$. Thus if $I_1=\emptyset$, then by~\eqref{equ12}
		and the fact that $|E(A_0,B)|\leq \binom{|I|}{2}$, we have
		$$\rho_{k,S}(A_0)\geq
		|I|(2(k+1)(k-2)-2(k-1))-\binom{|I|}{2}2(k-1)=|I|(2k^2-3k-3-|I|(k-1)).
		$$ The minimum of the last expression is achieved either for $|I|=2$
		or for $|I|=k$. If $|I|=2$, this is $2(2k^2-5k-1)>2(k+1)(k-2)-2(k-1)$.
		If $|I|=k$, this is $k(k^2-2k-3)$, which is again greater than
		$2(k+1)(k-2)-2(k-1)$. Thus $|I_1|\neq\emptyset$.
		
		Suppose $i,i'\in I_1$, $w\in A_i$, $w'\in A_{i'}$ and $ww'\in E(S)$.
		Let $A'=A_0-w-w'$. By the definition of $I_1$, all edges of $S[A_0]$
		incident with $w$ or $w'$ are in $E(B)$. Since $\Delta(B)\leq
		\frac{k}{2}$, $|E(S[A_0])|-|E(S[A'])|\leq k-1$. Thus
		$$\rho_{k,S}(A')\leq
		\rho_{k,S}(A_0)-2(k+1)(k-2)+(k-1)2(k-1)=\rho_{k,S}(A_0)-2k+6.$$ But
		by~\eqref{equ5'}, $|A'|\geq 2$, a contradiction to the minimality of
		$a$. It follows that for every $i\in I_1$, each neighbor in $A_0$ of
		the vertex $w\in A_i$ is in some $A_j$ for $j\in I_2$. This  implies
		$|E(A_0,B)|\leq \binom{|I|}{2}-\binom{|I_1|}{2}$. Together
		with~\eqref{equ11} and $\Delta(B)=\lfloor k/2\rfloor$, this yields
		that for each $i\in I_1$, the vertex $w\in A_i$ has exactly $k/2$
		neighbors in $B$, and all these neighbors are in $A$. In particular,
		$|I_2|\geq \frac{k}{2}$ and $k$ is even. Moreover, if $i,i'\in I_1$,
		$w\in A_i$ and $w'\in A_{i'}$, then their neighborhoods in $B$ are
		distinct, and thus in this case $|I_2|>\frac{k}{2}$. Since $k$ is
		even, this implies \begin{equation}\label{equ13} |I_2|\geq
		\frac{k+2}{2}. \end{equation}
		
		Since the potential of a single vertex is $(k+1)(k-2)$,
		\begin{equation}\label{equ14} \rho_{k,S}(A_0)\geq
		|I|(2(k+1)(k-2)-2(k-1))-|I_1|((k+1)(k-2)-2(k-1))-\left(\binom{|I|}{2}-\binom{|I_1|}{2}\right)2(k-1). \end{equation} The expression $-|I_1|((k+1)(k-2)-2(k-1)+\binom{|I_1|}{2}2(k-1)$ in~\eqref{equ14} decreases when $|I_1|$ grows but is at most $\frac{k-2}{2}$. Thus by~\eqref{equ13}, it is enough to let $|I_1|=|I|-\frac{k+2}{2}$ in~\eqref{equ14}. So, $$\rho_{k,S}(A_0)\geq |I|(k+1)(k-2)+\frac{k+2}{2}((k+1)(k-2)-2(k-1))-(k-1)(k+2)(|I|-\frac{k+4}{4}) $$ $$=-2k|I|+\frac{k+2}{2}\left[k^2-3k+\frac{k^2+3k-4}{2}\right]\geq -2k^2+\frac{(k+2)(3k^2-3k-4)}{4}>2(k+1)(k-2)-2(k-1) $$ for $k\geq 4$.\qed

		\begin{lemma}\label{cl4} Let $S=S_{k,t}$ with base $V_0$. Let
			$A\subseteq V(S)$ and $|A|\leq t+1$. \begin{equation}\label{equ13'}
			\mbox{If $ |A|\geq 2$, then   $\rho_{k,S}(A)\geq 2(k+1)(k-2)-2(k-1)$.}
			\end{equation} Moreover, \begin{equation}\label{equ14'} \mbox{if
				$A\supseteq V_0$, then $\rho_{k,S}(A)\geq 2(k+1)(k-2)$.}
			\end{equation} \end{lemma}
		
		{\bf Proof.} Suppose  the lemma is not true. Among $A\subseteq V(S)$
		with $|A|\geq 2$ for which~\eqref{equ13'} or~\eqref{equ14'} does not
		hold, choose $A_0$ with the smallest size. Let $a=|A_0|$. By
		Lemma~\ref{cl2}, $S[A_0]$ contains an edge $ww'$ in $B$. By
		Lemma~\ref{cl3}, $A_0$ contains a vertex  $v\in V_0$. In particular,
		$a\geq 3$.
		
		If $S[A_0]$ is disconnected, then $A_0$ is the disjoint union of
		nonempty $A'$ and $A''$ such that  $S$ has  no edges connecting $A'$
		with $A''$. Since $a\geq 3$, we may assume that $|A'|\geq 2$. By the
		minimality of $A_0$, $\rho_{k,S}(A')\geq 2(k+1)(k-2)-2(k-1)$. Also,
		$\rho_{k,S}(A'')\geq (k+1)(k-2)$. Thus
		$$\rho_{k,S}(A_0)=\rho_{k,S}(A')+\rho_{k,S}(A'')\geq
		2(k+1)(k-2)-2(k-1)+(k+1)(k-2)>2(k+1)(k-2), $$ contradicting the choice
		of $A_0$. Therefore, $S[A_0]$ {is connected.}
		
		Since the distance in $S$ between $v\in V_0$ and $\{w,w'\}\subset V(B)$
		is at least $t$, $a\geq t+2$, a contradiction.\qed

		\subsubsection*{Step 3: Completing the construction} Let
		$G=G(k,\epsilon)$ be obtained from a copy $H$ of $K_k$ by replacing
		every edge $uv$ in $H$ by a copy $S(uv)$ of $S_{k,t}$ with base
		$\{u,v\}$ so that all other vertices in these graphs are distinct.
		Suppose $G$ has a $(k-1)$-coloring $f$. Since $|V(H)|=k$, for some
		distinct $u,v\in V(H)$, $f(u)=f(v)$. This contradicts
		Lemma~\ref{lem2'}. Thus $\chi(G)\geq k$.
		
		Suppose there exists $A\subseteq V(G)$  with
		\begin{equation}\label{equ15} \mbox{$|A|\geq 2$ and
			$|E(G[A])|>1+(1+\epsilon)\frac{(k+1)(k-2)}{2(k-1)}(|A|-2)$.}
		\end{equation}
		
		Choose a smallest $A_0\subseteq V(G)$ satisfying~\eqref{equ15} and let
		$a=|A_0|$. Since a $2$-vertex (simple) graph has at most one edge,
		$a\geq 3$. We claim that \begin{equation}\label{equ16} \mbox{$G[A_0]$
			is $2$-connected.} \end{equation} Indeed, if not, then since $a\geq 3$,
		there are $x\in A_0$ and subsets $A_1,A_2$ of $A_0$ such that $A_1\cap
		A_2=\{x\}$, $A_1\cup A_2=A_0$, $|A_1|\geq 2$, $|A_2|\geq 2$, and there
		are no edges between $A_1-x$ and $A_2-x$ (this includes the case that
		$G[A_0]$ is disconnected). By the minimality of $a$, $|E(G[A_j])|\leq
		1+(1+\epsilon)\frac{(k+1)(k-2)}{2(k-1)}(|A_j|-2)$ for $j=1,2$. So,
		$$|E(G[A_0])|=|E(G[A_1])|+|E(G[A_2])|\leq
		2+(1+\epsilon)\frac{(k+1)(k-2)}{2(k-1)}(|A_1|+|A_2|-4) $$
		$$=2+(1+\epsilon)\frac{(k+1)(k-2)}{2(k-1)}(a-3)\leq
		1+(1+\epsilon)\frac{(k+1)(k-2)}{2(k-1)}(a-2),$$
		contradicting~\eqref{equ15}. This proves~\eqref{equ16}.
		
		Let $J=\{uv\in E(H)\,:\,A_0\cap (V(S(uv)-u-v)\neq\emptyset\}$. For
		$uv\in J$, let $A_{uv}=A_0\cap (V(S(uv))$. Since $G[A_0]$ is
		$2$-connected, for each $uv\in J$, \begin{equation}\label{equ17}
		\mbox{$\{u,v\}\subset A_{uv}$ and $G[A_{uv}]$ is connected. In
			particular, $|A_{uv}|\geq 4$.} \end{equation}
		
		Our next claim is that for each $uv\in J$,
		\begin{equation}\label{equ18} \mbox{ $|E(G[A_{uv}])|\leq
			(1+\epsilon)\frac{(k+1)(k-2)}{2(k-1)}(|A_{uv}|-2)$.} \end{equation}
		Indeed, if $|A_{uv}|\leq t+1$, this follows from Lemma~\ref{cl4}. If
		$|A_{uv}|\geq t+2$, then by the part of Lemma~\ref{cl2} dealing with
		$A\supseteq V_0$, $$|E(G[A_{uv}])|\leq
		|E(B_{k,t})|+\frac{(k+1)(k-2)}{2(k-1)}(|A_{uv}|-2)=\binom{k}{2}+\frac{(k+1)(k-2)}{2(k-1)}(|A_{uv}|-2). $$ But since $t\geq k^3/\epsilon$, $\binom{k}{2}<\epsilon t\frac{(k+1)(k-2)}{2(k-1)}$. This proves~\eqref{equ18}.
		
		By~\eqref{equ18}, \begin{equation}\label{equ19} |E(G[A_0])|=\sum_{uv\in
			J}|E(G[A_{uv}])|\leq (1+\epsilon)\frac{(k+1)(k-2)}{2(k-1)}\sum_{uv\in
			J}(|A_{uv}|-2) \end{equation} Since each $A_{uv}$ has at most two
		vertices in common with the union of all other $A_{u'v'}$, $\sum_{uv\in
			J}(|A_{uv}|-2)\leq a-2$. Thus~\eqref{equ19} contradicts the choice of
		$A_0$. It follows that no $A\subseteq V(G)$ satisfies~\eqref{equ15},
		which exactly means that $m_2(G)\leq
		(1+\epsilon)\frac{(k+1)(k-2)}{2(k-1)}$. \end{proof}

	\section{The case $m_2(H)<m_2(K_m)$} \label{sec:DifOrder}

	When $m_2(H)<m_2(K_m)$ we show that as in the previous case there are
	two typical behaviors of the function $ex(G(n,p),K_m,H)$. For small
	values of $p$ Lemma \ref{lem:SmallValuesOfp} shows that there exists
	w.h.p.\ an $H$-free subgraph of $G(n,p)$ which contains all but a
	negligible part of the copies of $K_m$. For large values of $p$ Lemma
	\ref{lem:BigValuesOfp} shows that w.h.p.\ every $H$-free graph will have
	to contain a much smaller proportion of the copies of $K_m$.
	
	However, unlike in the case $m_2(H)>m_2(K_m)$ discussed in Section
	\ref{sec:EasyOrder}{,}   the change between the behaviors for $p=n^{-a}$
	does not happen at $-a=-1/m_2(H)$. Theorem \ref{thm:pAlmostM_2(H)} shows
	that if $p=n^{-a}$ and $-a$ is slightly bigger than $-1/m_2(H)$ we can
	still take all but a negligible number of copies of $K_m$ into an
	$H$-free subgraph. 
	As for a conjecture about where the change happens (and if there are
	indeed two regions of different behavior and not more) see the
	discussion in the last section.

	\begin{proof} [Proof of Theorem \ref{thm:pAlmostM_2(H)}] Let $G\sim
		G(n,p)$ with $p=n^{-a}$ where $-a=-c+\delta$ for some small $\delta>0$
		to be chosen later. Let $G'$ be the graph obtained from $G$ by first
		removing all pairs of copies of $K_m$ sharing an edge and then removing
		all edges that do not take part in a copy of $K_m$. As $\delta$ is
		small, we may assume that  $-a<-1/m_2(K_m)$, apply Lemma
		\ref{lem:oneK_mPerEdge} and deduce that w.h.p.\ the number of copies of
		$K_m$ removed in the first step is $o(\binom{n}{m}p^{\binom{m}{2}})$. In
		the second step there are no copies of $K_m$ removed, and thus w.h.p.
		$\mathcal{N}(G,K_m)=(1+o(1))\mathcal{N}(G',K_m)$. \color{black}
		Furthermore, if there is a copy of $H_0$ in $G'$ then each edge of it
		must be contained in a copy of $K_m$ and not in two or more such copies.
		\color{black}
		
		Let $\mathcal{H}_m$ be the family of the following graphs. Every graph
		in $\HY_m$ is an edge disjoint union of copies of $K_m$, it contains a
		copy of $H_0$ and removing any copy of $K_m$ makes it $H_0$-free. Note
		that if $G$ is $\mathcal{H}_m$-free then $G'$ is $H_0$-free.
		
		To show that $G$ is indeed $\HY_m$-free w.h.p.\ we prove that for any
		$H'\in \mathcal{H}_m$ the expected number of copies of it in $G$ is
		$o(1)$. We will show this for  $p=n^{-\frac{1}{m_2(H)}+\delta}$, and it
		will thus clearly hold for smaller values of $p$ as well. For every
		$H'$ the expected number of copies of it in $G(n,p)$ is
		$\Theta(p^{e(H')}n^{v(H')})=
		\Theta(n^{-\frac{1}{m_2(H)}e(H')+v(H')}n^{\delta\cdot e(H')})$ and we
		want to show that it is equal $o(1)$ for any $H'$. For this it is
		enough to show that
		$-\frac{e(H')}{v(H')}+m_2(H)-\delta\frac{e(H')}{v(H')}m_2(H)<0$. We
		first prove that $$d(H'):=\frac{e(H')}{v(H')}>m_2(H)+\delta'$$ {f}or
		some $\delta':=\delta'(m,c)$ and then to finish show that
		$\frac{e(H')}{v(H')}m_2(H)\leq g(m)$ for some function $g$. 

		Note that every $H'\in \mathcal{H}_m$ contains a copy of $H_0$ and that
		$H_0$ itself does not contain a copy of $K_m$ as $m_2(H_0)<m_2(K_m)$.
		The vertices of copies of $K_m$ in $H'$ can be either all from $H_0$ or
		use some external vertices. Let $E_1$ be the edges between two vertices
		of $H_0$ that are not part of the original $H_0$ and let $|E_1|=e_1$.
		Furthermore, let $V_1\cup ...\cup V_k=V(H')\setminus V(H_0)$ be the
		external vertices, where each $V_i$ creates a copy of $K_m$ with the
		other vertices from $H_0$ and let $|V_i|=v_i$.
		
		Each edge in $H_0$ must be {a} part of a copy of $K_m$. An edge in
		$E_1$ takes care of at most $\binom{m}{2}-1$ edges from $H_0$, and each
		$V_i$ takes care of at most $\binom{m-v_i}{2}$ edges. From this we get
		that
		
		\begin{align*} e(H_0)\leq& \sum_{i=1}^{k} \binom{m-v_i}{2} +e_1
		(\binom{m}{2}-1) \\ \leq & k\binom{m-1}{2}+e_1(\binom{m}{2}-1) \\ \leq
		& \frac{m^2}{2}(k+e_1){.} \end{align*}
		
		We will take care of two cases, either $e_1\geq \frac{e(H_0)}{m^2}$ or
		$k\geq \frac{e(H_0)}{m^2}$. In the first case let $H_1$ be the graph
		$H_0$ together with the edges in $E_1$. Then $$
		\frac{e(H_1)}{v(H_1)}=\frac{e(H_0)+e_1}{v(H_0)}\geq
		(1+\frac{1}{m^2})\frac{e(H_0)}{v(H_0)}{.} $$ We can assume $v(H_0)$ is
		large enough so that
		$\frac{e(H_0)}{v(H_0)}/\frac{e(H_0)-1}{v(H_0)-2}\geq
		(1-\frac{1}{2m^2})$ and as $m_2(H_0)$ is bounded from below by a
		function of $m${,} we get that for some $\delta':=\delta'(m)$ small
		enough we get $$\frac{e(H_1)}{v(H_1)}\geq m_2(H_0)+\delta'{.}$$ {H}ence
		w.h.p.\ there is no copy of $H_1$ in $G$, and thus no copy of $H'$.
		
		Now let us assume that $k\geq \frac{e(H_0)}{m^2}$ and let $\gamma =
		m_2(K_m)-m_2(H)\geq m_2(K_m)-c$. The expression
		$\frac{\binom{v_i}{2}+v_i(m-v_i)}{v_i}$ decreases with $v_i$, and as
		$V_i$  creates a copy of $K_m$ with an edge of $H_0${,} we get that
		$v_i\leq m-2$ and so $\frac{\binom{v_i}{2}+v_i(m-v_i)}{v_i}\geq
		\frac{\binom{m}{2}-1}{m-2}$. It follows that
		
		\begin{equation} \label{eq:30} \sum_{i=0}^{k} \binom{v_i}{2} +
		v_i(m-v_i)\geq \sum_{i=0}^k v_i
		\frac{\binom{m}{2}-1}{m-2}=\sum_{i=0}^{k}v_i (m_2(H_0)+\gamma){.}
		\end{equation}
		
		Every set of vertices $V_i$ uses at least one edge in $H_0$ for a copy
		of $K_m$, and as there are no two copies of $K_m$ sharing an edge{,}
		\color{black} it follows that:\color{black} \begin{equation*}
		v(H')=v(H_0)+\sum_{i=0}^{k}v_i \leq e(H_0)+(m-1)e(H_0)=m\cdot e(H_o){.}
		\end{equation*} {C}ombining this with the assumption on $k$ we
		\color{black}conclude\color{black} \begin{equation}\label{eq:31}
		\sum_{i=0}^k v_i \geq k \geq \frac{e(H_0)}{m^2}\geq \frac{v(H')}{m^3}.
		\end{equation} Finally a direct calculation {yields} \begin{equation}
		\label{eq:32} e(H_0)+e_1> e(H_0)-1=\frac{e(H_0)-1}{v(H_0)-2}
		(v(H_0)-2)=m_2(H_0)(v(H_0)-2){.} \end{equation}

		Applying the above inequalities we get
		
		\begin{align*}
		e(H')=&e(H_0)+e_1+\sum_{i=0}^{k}\binom{v_i}{2}+v_i(m-v_i)\\
		\overset{\ref{eq:30},\ref{eq:32}}{\mbox{ }\geq}& m_2(H)(\sum_{i=0}^{k}
		v_i + v(H_0)-2)+\sum_{i=0}^{k} v_i \gamma \\ =
		&m_2(H)(v(H')-2)+\sum_{i=0}^{k} v_i \gamma \\ 
		\overset{\ref{eq:31}}{\geq} & m_2(H)(v(H')-2)+\frac{v(H')}{m^3}\gamma\\
		\geq& v(H')(m_2(H_0)+\frac{1}{2m^3}\gamma){.} \end{align*} The last
		inequality \color{black}holds \color{black} if $2m_2(H)\leq
		v(H')\frac{\gamma}{2m^3}${,} but this {is} true as $v(H_0)$ is large
		enough. \color{black} Thus, \color{black} for $\delta':=\delta'(m,c)$
		small enough{,} $$\frac{e(H')}{v(H')}\geq
		m_2(H_0)+\frac{1}{2m^3}\gamma\geq
		m_2(H_0)+\frac{1}{2m^3}(m_2(K_m)-c)\geq m_2(H_0)+\delta' $$ and again,
		w.h.p. $G$ will not have a copy of $H'$.
		
		It is left to show that indeed $\frac{e(H')}{v(H')}m_2(H)\leq g(m)$.
		\color{black}By \color{black} the definition of $H'$ we get that
		$\frac{e(H')}{v(H')}< \frac{e(H_0)+(m-2)e(H_0)}{v(H')}
		=(m-1)\frac{e(H_0)}{v(H_0)}${.  A}s we may assume that $v(H_0)$ is
		large{,} it follows that $\frac{e(H_0)}{v(H_0)}\leq
		m_2(H_0)(1+\frac{1}{m})$, and as $m_2(H)<m_2(K_m)${,} \color{black} we
		conclude \color{black} that for some $g(m)$ the needed inequality holds.
		
		%
		%
		%
		%
		%
		
	\end{proof}
	
	To finish this section, {we} show that indeed the theorem can be applied
	to  $G(m+1,\epsilon)$.
	
	\begin{proof}[Proof of Lemma \ref{lem:bigV}] 
		To prove this we will use the following fact. If $\frac{a}{b}$ and
		$\frac{p}{q}$ are rational numbers such that
		$0<|\frac{a}{b}-\frac{p}{q}|\leq \frac{1}{bM}$  then $p\geq M$. Indeed,
		assume towards {a} contradiction that $q<M$\color{black}, \color{black}
		but then $|\frac{a}{b}-\frac{p}{q}|=|\frac{aq-bp}{bq}|\geq
		\frac{1}{bq}> \frac{1}{bM}$.
		
		Let $G_0:=G_0(m+1,\epsilon)$, and take
		$\frac{a}{b}=\frac{(m+2)(m-1)}{2m}$ and
		$\frac{p}{q}=\frac{e(G_0)-1}{v(G_0)-2}$. \color{black}By \color{black}
		Theorem \ref{thm:existOfH} \color{black} it follows \color{black} that
		$|\frac{a}{b}-\frac{p}{q}|\leq \epsilon \frac{(m+2)(m-1)}{2m}$.
		Choosing $\epsilon$ small enough will make $v(G_0)$ as large as needed.

		%
	\end{proof}

	\section{Concluding remarks and open problems} 
	\begin{itemize} \item It is interesting to note that there are two main
		behaviors of the function $ex(G(n,p),K_m,H)$ that we know of. For $K_m$
		and $H$ with $\chi(H)=k>m$  for small $p$ one gets that an $H$-free
		subgraph of $G\sim G(n,p)$ can contain w.h.p.\ most of the copies of
		$K_m$ in the original $G$. On the other hand, when $p >\max \{
		n^{-1/m_2(H)},n^{-1/m_2(K_m)} \}$ then an $H$-free graph with the
		maximal number of $K_m$s is essentially w.h.p.\ $k-1$ partite, thus has
		a constant proportion less copies of $K_m$ than $G$.
		
		If $m_2(H)>m_2(K_m)$ then Theorem \ref{thm:m_2InEasyOrder} shows that
		the behavior changes at $p=n^{-1/m_2(H)}$, but if $m_2(H)<m_2(K_m)$ the
		critical value of $p$ is bounded away from $n^{-1/m_2(H)}$ and it is
		not clear where exactly it is.
		
		Looking at the graph $G\sim G(n,p)$ and taking only edges that take
		part in a copy of $K_m$ yields another random graph $G|_{K_m}$. The
		probability of an edge to take part in $G|_{K_m}$ is  $\Theta(p\cdot
		n^{m-2}p^{\binom{m}{2}-1})$. A natural conjecture is that if
		$n^{m-2}p^{\binom{m}{2}}$ is much bigger than $n^{-1/m_2(H)}$ then when
		maximizing the number of $K_m$ in an $H$-free subgraph we cannot avoid
		a copy of $H$ by deleting a negligible number of copies of $K_m$ and
		when $n^{m-2}p^{\binom{m}{2}}$ is much smaller than $n^{-1/m_2(H)}$ we
		can keep most of the copies of $K_m$ in an $H$-free subgraph of $G\sim
		G(n,p)$. It would be interesting to decide if this is indeed the case.
		
		\item Another possible model of a random graph, tailored specifically
		to ensure that each edge lies in a copy of $K_m$, is the following.
		Each $m$-subset of a set of $n$ labeled vertices, randomly and
		independently, is taken as an $m$-clique with probability $p(n)$. In
		this model the resulting random  graph $G$  is equal to its subgraph
		$G|_{K_m}$ defined in the previous paragraph, and  one can study the
		behavior of the maximum possible number of copies of $K_m$ in an
		$H$-free subgraph of it for all admissible values of $p(n)$.

		\item There are other graphs $T$ and $H$ for which $ex(n,T,H)$ is
		known, and one can study the behavior of $ex(G(n,p),T,H)$ in these
		cases. For example in \cite{HHKNR} and independently in \cite{G}  it is
		shown that $ex(n,C_5,K_3)=(n/5)^5$ when $n$ is divisible by 5.
		
		Using some of the techniques in this paper we can prove that for
		\color{black} $p\gg n^{-1/2}=n^{-1/m_2(K_3)}$, 
		$ex(n,C_5,K_3)=(1+o(1))(np/5)^5$ w.h.p.\ whereas if $p\ll n^{-1/2}$
		\color{black} then w.h.p. $ex(n,C_5,K_3)=(\frac{1}{10}+o(1))(np)^5$.
		Similar results can be proved in additional cases for which
		$ex(n,T,H)=\Omega(n^{t})$ where $t$ is the number of vertices of $T$.
		As observed in \cite{ASh}{,} these are exactly all pairs of graphs
		$T,H$ where $H$ is not a subgraph of any blowup of $T$.
		
		\item When investigating $ex(G(n,p),T,H)$ here we focused on the case
		that $T$ is a complete graph. It is possible that a variation
		\color{black} of \color{black} Theorem \ref{thm:m_2InEasyOrder} can be
		proved for any $T$ and $H$ satisfying $m_2(T)>m_2(H)$, even without
		knowing the exact value of $ex(n,T,H)$. %
		
		\item In the cases studied here for non-critical values of $p$,
		$ex(G(n,p),T,H)$ is always either almost all copies of $T$ in $G(n,p)$
		or $(1+o(1))ex(n,T,H)p^{e(T)}$. It would be interesting to decide if
		such a phenomenon holds for all $T$, $H$.
		
		\item As with the classical Tur\'an problem, the question studied here
		can be investigated for a  general graph $T$ and finite or infinite
		families $\mathcal{H}$. 
		
	\end{itemize}
	
	
	\section*{Acknowledgment} \color{black} We thank an anonymous referee
	for valuable and helpful comments. \color{black}

\end{document}